\documentclass{amsart} 
\usepackage{amssymb, amsmath}
\usepackage[dvips]{graphicx}
\usepackage{amsfonts}
\usepackage{latexsym}
\usepackage{color}
\newtheorem{theorem}{Theorem}	
\newtheorem{lemma}{Lemma}[section]		
		
\newtheorem{proposition}{Proposition}		
			
\newtheorem{definition}{Definition}	

\newtheorem{claim}{Cliam}[section]
\title[The spherical dual transform is an isometry]
{
The spherical dual transform is an isometry \\ 
for spherical Wulff shapes 
}
\thanks{\color{black} This work is partially supported by JSPS KAKENHI Grant Number 26610035.}
\author{
Huhe Han}
\address{Graduate School of Environment and Information Sciences,Yokohama National University, {\color{black}Yokohama 240-8501,} Japan}
\email{han-huhe-bx@ynu.jp}
\author{Takashi Nishimura
}
\address{Research Group of Mathematical Sciences,  
Research Institute of Environment and Information Sciences,  
Yokohama National University, 
Yokohama 240-8501, Japan}
\email{nishimura-takashi-yx@ynu.jp}
\begin{document}
\begin{abstract}
{\color{black}
A spherical Wulff shape is the spherical counterpart of a Wulff shape which is the well-known geometric model 
of a crystal at equilib{\color{black} r}ium introduced by G. Wulff in 1901.    
As same as a Wulff shape, each spherical Wulff shape has its unique dual.    
The spherical dual transform for spherical Wulff shapes 
is the mapping which maps a spherical Wulff shape to its spherical dual {\color{black}Wulff shape}.   
In this paper, it is shown that the spherical dual transform {\color{black}for spherical Wulff shapes} 
is an isometry with respect to the Pompeiu-Hausdorff metric.  
} 
%
\end{abstract}

\maketitle  

\section{Introduction}
\par
{\color{black} 
Throughout this paper, 
let $n$ (resp., $S^n$) be a positive integer 
(resp., the unit sphere in $\mathbb{R}^{n+1}$).    
For any point $P\in S^n$, let $H({\color{black}P})$ 
be the closed hemisphere centered at $P$, 
namely, $H(P)$ is the set consisting of $Q\in S^n$ 
satisfying $P\cdot Q\ge 0$, 
where the dot in the center stands for the scalar product of two vectors 
$P, Q\in \mathbb{R}^{n+1}$.     
For any non-empty subset $W\subset S^n$, 
the {\it spherical polar set of $W$}, denoted by 
$W^\circ$, is defined as follows: 
\[
W^\circ = \bigcap_{P\in W}H(P).
\]   
\par 
In \cite{nishimurasakemi2}, the spherical polar set plays an essential role 
for investigating a Wulff shape, which is 
the geometric model of a crystal at equilibrium 
introduced by G.~Wulff in \cite{wulff}. 
\par 
Let $\mathcal{H}(S^n)$ be the set 
consisting of non-empty closed subsets of 
$S^n$.    
It is well-known that $\mathcal{H}(S^n)$ is a complete metric space 
with respect to 
the Pompeiu-Hausdorff metric (for instance, see \cite{barnsley, falconer}). 
Let $\mathcal{H}^\circ (S^n)$ be 
the {\color{black}subspace} of $\mathcal{H}(S^n)$ consisting of 
non-empty closed subset $W$ of $S^n$ 
such that $W^\circ\ne \emptyset$.  
The {\it spherical polar transform} 
$\bigcirc: \mathcal{H}^\circ(S^n)\to \mathcal{H}^\circ(S^n)$ is defined by 
$\bigcirc(W)=W^\circ$.
Since $W\subset W^{\circ\circ}$ for any $W\in \mathcal{H}^\circ(S^n)$ 
by Lemma 2.2 of \cite{nishimurasakemi2}, it follows that 
$W^\circ\in \mathcal{H}^\circ(S^n)$ 
for any $W\in \mathcal{H}^\circ(S^n)$.   
Thus, the spherical polar transform $\bigcirc$ is well-defined
\footnote{\color{black} 
Since $(S^n)^\circ =\emptyset$ for any $n\in \mathbb{N}$, 
the spherical polar transform defined in \cite{aperture} 
should be understood as 
$\bigcirc: \mathcal{H}^\circ (S^n)\to \mathcal{H}^\circ(S^n)$.}.      
\par 
In \cite{aperture}, crystal growth is investigated 
by introducing a geometric model of 
a certain growing crystal in $\mathbb{R}^2$.     
One of the powerful tools in \cite{aperture} is  
the {spherical polar transform} 
$\bigcirc: \mathcal{H}^{\circ}(S^{2})\to \mathcal{H}^{\circ}(S^{2})$.     
Especially, for studying the dissolving process 
of the geometric model introduced in \cite{aperture}, 
the spherical polar transform {\color{black}is} indispensable 
since it enables one to analyze in detail  
the image of dissolving one-parameter family of spherical 
Wulff shapes.      
Notice that it is impossible to give a similar analysis in the 
Euclidean space $\mathbb{R}^n$ since the corresponding image 
in $\mathbb{R}^n$ is divergent.      
We consider that 
the spherical polar transform may be more applicable 
{\color{black}especially} 
to investigate 
dissolving process of Wulff shapes 
and that it is significant to obtain useful properties 
of the spherical polar transform.   
\par 
In this paper, motivated by these considerations, 
we investigate natural restrictions of the spherical polar transform.   
{\color{black}
The most natural subspace for the restriction 
of spherical polar transform is 
\[\mathcal{H}_{\rm Wulff}(S^n, P)
\] 
defined 
as follows.   
\begin{definition}\label{definition closure}
{\rm 
\begin{enumerate}
\item\quad 
Let $W$ be a subset of $S^n$.    
Suppose that there exists a point $P\in S^n$ such that 
$W\cap H(P)=\emptyset$.   Then, $W$ is said to be 
{\it hemispherical}.   
\item\quad 
Let $W\subset S^{n}$ be a hemispherical subset.   
Let $P, Q$ be two points of $W$.     
Then, 
the following arc is denoted by $PQ$:  
\[
PQ=\left\{\left.\frac{(1-t)P+tQ}{||(1-t)P+tQ||}\in S^{n}\; \right|\; 0\le t\le 1\right\}.  
\]
\item\quad 
Let $W\subset S^{n}$ be a hemispherical subset.   
Suppose that 
$PQ\subset W$ for any $P, Q\in W$.     Then,  
$W$ is said to be {\it spherical convex}.   
{\color{black}
\item\quad  
Let $W\subset S^{n}$ be a hemispherical subset.   
Suppose that $W$ is closed, spherical convex and has an interior point.   
Then,  
$W$ is said to be a {\it spherical convex body}.   
}
\item
\quad 
For any point $P$ of $S^n$, 
let $\mathcal{H}_{\rm Wulff}(S^n, P)$ be the following set: 
\begin{eqnarray*}
\mathcal{H}_{\rm Wulff}(S^n, P) & = &  
\left\{
W\in \mathcal{H}(S^n)\; \right| 
\; W\cap H(-P)=\emptyset, P\in \mbox{int}(W),  \\ 
{ } & { } & \qquad\qquad  
\left. W\mbox{ is a spherical convex body}
\right\}, 
\end{eqnarray*}
where $\mbox{int}(W)$ stands for the set consisting of interior points of $W$.        
The topological closure of 
$\mathcal{H}_{\rm Wulff}(S^n, P)$ is denoted by $\overline{\mathcal{H}_{\rm Wulff}(S^n, P)}$. 
\item\quad 
For any $P\in S^n$, an element of $\mathcal{H}_{\rm Wulff}(S^n, P)$ 
is called a {\it spherical Wulff shape}.
\end{enumerate} 
}
\end{definition}
\noindent 
It is known that a Wulff shape in $\mathbb{R}^n$ can be characterized 
as a convex body of $\mathbb{R}^n$ such that the origin is an interior point of 
it, namely, as a compact and convex subset of $\mathbb{R}^n$ 
such that the origin is an interior point of it (\cite{taylor}).   
Hence, the definition of spherical Wulff shape is reasonable.   
The restriction of $\bigcirc$  
to $\mathcal{H}_{\rm Wulff}(S^{n}, P)$ 
(resp., $\overline{\mathcal{H}_{\rm Wulff}(S^{n}, P)}$) 
is 
called the {\it spherical dual transform relative to $P$} 
(resp., {\it extension of the spherical dual transform relative to $P$}) 
and is denoted by 
$\bigcirc_{{\rm Wulff},P}$ (resp., $\overline{\bigcirc_{{\rm Wulff},P}}$).   
The set $\bigcirc(W)=W^\circ$ is called the 
{\it spherical dual Wulff shape} of $W$ 
if $W$ is a spherical Wulff shape.      
Thus, it is reasonable to call $\bigcirc_{{\rm Wulff},P}$ 
the spherical dual transform.       
It is not difficult to have the following  
(cf. {\color{black}Proposition \ref{lemma 1.1}} in Subsection 
\ref{subsection 5.1}). 
{\color{black}
\begin{eqnarray*}
\bigcirc_{{\rm Wulff},P}: 
\mathcal{H}_{\rm Wulff}(S^{n}, P) 
& \to & 
\mathcal{H}_{\rm Wulff}(S^{n}, P) 
\mbox{ is well-defined and bijective}, \\     
\overline{\bigcirc_{{\rm Wulff},P}}: 
\overline{\mathcal{H}_{\rm Wulff}(S^{n}, P)}
& \to & 
\overline{\mathcal{H}_{\rm Wulff}(S^{n}, P)}   
\mbox{ is well-defined and bijective}. 
\end{eqnarray*}
}
\par 
The main purpose of this paper is to show the following:   
\begin{theorem}\label{theorem 1}
Let $P$ be a point of $S^n$.   {\color{black}
Then, with respect to the Pompeiu-Hausdorff metric, 
the following two hold:   
\begin{enumerate}
\item {\color{black}The spherical dual transform relative to $P$} 
\[
\bigcirc_{{\rm Wulff},P} : 
{\color{black}{\mathcal{H}_{\rm Wulff}(S^{n}, P)}}\rightarrow  
{\color{black}{\mathcal{H}_{\rm Wulff}(S^{n}, P)}}
\]
is an isometry. 
\item The extension of spherical dual transform 
relative to $P$ 
\[
\overline{\bigcirc_{{\rm Wulff},P}} : 
{\color{black}\overline{\mathcal{H}_{\rm Wulff}(S^{n}, P)}}\rightarrow  
{\color{black}\overline{\mathcal{H}_{\rm Wulff}(S^{n}, P)}}
\]
is an isometry. 
\end{enumerate}
}
\end{theorem}
\noindent 
For any positive real number $r$, 
let $D_r$ be the set consisting of $x\in \mathbb{R}^n$ satisfying 
$||x||\le r$. Then, $D_r$ is a Wulff shape 
for any $r\in \mathbb{R}$ $(r>0)$ and 
it is well-known that the dual Wulff shape of $D_r$ is $D_{\frac{1}{r}}$.   
Moreover, it is easily seen that 
$h(D_{r_1}, D_{r_2})=|r_1-r_2|$  holds 
for any $r_1, r_2\in \mathbb{R}$  ($r_1, r_2>0$), 
{\color{black}where $h$ is the Pompeiu-Hausdorff metric}.   
Thus, it is impossible to expect the Euclidean counterpart 
of the assertion (1) of Theorem \ref{theorem 1}.   
This shows an advantage of studying the spherical version of Wulff shapes.   
Moreover, the Euclidean counterpart of the extension 
of spherical dual transform 
relative to $P$ is not well-defined.   
This, too, shows an advantage of 
studying the spherical version of Wulff shapes.       
}
\par 
\medskip 
Next, we investigate the restriction of $\bigcirc$ to   
\[
\overline{\mathcal{H}_{\mbox{s-conv}}(S^n)}, 
\] 
which is the topological closure of 
the set consisting of spherical convex closed subsets. 
The restriction of $\bigcirc$  
to $\overline{\mathcal{H}_{\mbox{s-conv}}(S^n)}$ 
is 
denoted by 
$\bigcirc _{\mbox{\rm s-conv}}$.     
{\color{black} 
It is not hard to see the following 
(cf. Proposition \ref{lemma 1.1} in Subsection \ref{subsection 5.1}).   
\[
\bigcirc _{\mbox{\rm s-conv}}: 
\overline{\mathcal{H}_{\mbox{s-conv}}(S^n)} 
\to 
\overline{\mathcal{H}_{\mbox{s-conv}}(S^n)} 
\mbox{ is well-defined and bijective}. 
\]   
\begin{theorem}\label{theorem 2}
With respect to the Pompeiu-Hausdorff metric, 
the restriction of the spherical polar transform 
\[
\bigcirc _{\mbox{\rm s-conv}} : 
\overline{\mathcal{H}_{\mbox{{\rm s-conv}}}(S^{n})}
\rightarrow \overline{\mathcal{H}_{\mbox{{\rm s-conv}}}(S^{n})}
\] 
is bi-Lipschitz but never an isometry. 
\end{theorem}
\bigskip 
This paper is organized as follows. 
In Section \ref{section 2}, preliminaries for the proofs of Theorems 
\ref{theorem 1} and \ref{theorem 2} 
are given.       
Theorems \ref{theorem 1} and \ref{theorem 2} are proved in 
Sections \ref{section 3} and \ref{section 4} respectively.  
Section \ref{section 5} is an appendix where,   
for the sake of readers' convenience, 
it is proved that 
all of $\bigcirc_{{\rm Wulff},P}$, 
$\overline{\bigcirc_{{\rm Wulff},P}}$ and 
$\bigcirc _{\mbox{\rm s-conv}}$ 
are well-defined bijective mappings; and moreover it is explained why 
{\color{black}
the restriction of $\bigcirc$ to 
$\overline{\mathcal{H}_{\mbox{{\rm s-conv}}}(S^{n})}$ 
is important.}   
\section{Preliminaries}\label{section 2}
\subsection{Convex geometry in $S^n$}\label{convex in sphere}
\begin{definition}[\cite{nishimurasakemi2}]\label{definition 2.2}
{\rm Let $W$ be a hemispherical subset of $S^{\color{black}n}$.     
Then, the following set, denoted by 
$\mbox{\rm s-conv}({\color{black}W})$, is called the {\it spherical convex hull of} 
${\color{black}W}$.   
\[
\mbox{\rm s-conv}({\color{black}W})= 
\left\{\left.
\frac{\sum_{i=1}^k t_iP_i}{||\sum_{i=1}^kt_iP_i||}\;\right|\; 
P_i\in {\color{black}W},\; \sum_{i=1}^kt_i=1,\; t_i\ge 0, k\in \mathbb{N}
\right\}.
\] 
}
\end{definition}
{\color{black} \begin{lemma}[\cite{nishimurasakemi2}]\label{lemma 3}
Let $W_{1}, W_{2}$ be non-empty subsets of $S^{n}$.    
Suppose that the inclusion ${W_{1}\subset W_{2}}$ holds.    
Then, the inclusion $W_{2}^{\circ}\subset W_{1}^{\circ}$ holds.
\end{lemma}}
\begin{lemma}[\cite{nishimurasakemi2}]\label{lemma 4}
For any non-empty closed hemispherical subset $X \subset S^{n}$, 
{\color{black}
the equality $\mbox{ \rm s-conv}(X)= \left(
\mbox{ \rm s-conv}\left( X \right) \right)^{\circ\circ }$ holds.
} 
\end{lemma}
The following proposition may be regarded as a spherical version of 
the separation theorem, which may be easily obtained by 
the separation theorem in Euclidean space 
(for the separation theorem in Euclidean space, 
see for instance \cite{matousek}).   
\begin{proposition}\label{proposition matousek}
Let $P$ be a point of $S^n$ 
and let $W_1, W_2$ be {\color{black} closed} spherical convex sets 
such that $W_i\cap H(P)=\emptyset$ $(i=1,2)$.     
Suppose that $W_1\cap W_2=\emptyset$.   
Then, there exists a point $Q\in S^n$ satisfying the following:   
\[
W_1\subset H(Q)\quad \mbox{and}\quad W_2\cap H(Q)=\emptyset.   
\]         
\end{proposition} 
\subsection{Metric geometry in $S^n$}\label{metric in sphere}
\quad { }
\par 
{\color{black}
For any $P, Q\in S^{n}$, 
the length of the arc $ PQ$ is denoted by $|PQ|$.
}
\begin{lemma}\label{lemma 1}
For any $P, Q\in S^{n}$ such that $|PQ|\leq \frac{\pi }{2}$, 
the following equality holds: 
\[
h(H(P),H(Q))=
|PQ|.
\]
\end{lemma}
\begin{proof}\quad 
{\color{black}By} Figure 1, 
it is clear that $h(H(P),H(Q))+r=\mid PQ\mid+\ r=\frac{\pi }{2}$, 
so we have $h(H(P),H(Q))=
\mid PQ\mid${\color{black}.}  
\end{proof}
\begin{figure}[hbtp]
\begin{center}
\includegraphics[width=4cm]{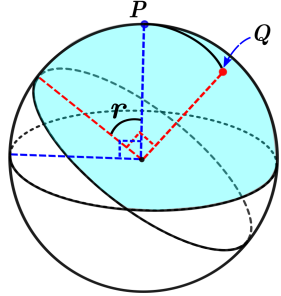}
\caption{$h(H(P),H(Q))= |PQ|$.}
\label{figure 1}
\end{center}
\end{figure}   
\begin{definition}\label{definition 2.2}
{\rm 
\begin{enumerate}
\item 
For any point $P\in S^n$ and any real number $r$ satisfying $0<r<\pi$, 
define the following two sets:
\begin{eqnarray*}    
\overline{B(P, r)} & = & \{Q\in S^{n}\; |\; |PQ|\leq r\},  \\ 
\partial \overline{B(P,r)} & = & {\{Q\in S^{n}\mid |PQ|=r\}}.
\end{eqnarray*}
{\color{black}
\item 
For any non-empty subset $W\subset S^n$ and 
any real number $r$ satisfying $0<r<\pi$, define the following two sets:
\[    
\overline{B(W, r)}  =  \bigcup_{P\in W}\overline{B(P, r)}.  
\]
}
\end{enumerate}
}
\end{definition}
\begin{lemma}\label{lemma 2}
For any subset $W\in S^{n}$ such that $W^{\circ}$ is spherical convex set 
and any real number $r$ satisfying $0<r<\frac{\pi}{2}$, 
the following equality holds:
\[
\overline{B\left(\displaystyle\bigcap_{P\in W}H(P),r\right)}=\displaystyle\bigcap_{P\in W}\overline{B(H(P), r)}.
\]
\end{lemma}
\par
\begin{proof} \indent 
''$\subset$ '' \indent
Let $Q$ be a point of $\overline{B\left(\bigcap_{P\in W}H(P),r\right)}$. 
Then, it follows that 
$\overline{B(Q, r)}\cap \bigcap_{P\in W}H(P)\neq \emptyset$. 
Thus, there exists a point $ Q_{1}\in \overline{B(Q, r)}$ 
such that $Q_{1}$ belongs to $\bigcap_{P\in W}H(P)$.    
Therefore, there exists a point $Q_{1}\in \overline{B(Q, r)}$ 
such that $Q_{1}\in H(P) $ for any $P\in W$.    
It follows that $Q\in \bigcap_{P\in W}\overline{B(H(P), r)}.$\\
{\color{black} ''$\supset$'' 
\indent 
Suppose that 
there exists a point $Q \in \bigcap_{P\in W}\overline{B(H(P), r)}$ 
such that $Q\notin \overline{B(\bigcap_{P\in W}H(P), r)}$.    
Since $Q$ belongs to $\bigcap_{P\in W} \overline{B(H(P), r)}$, 
we have that $Q\in \overline{B(H(P), r)}$ for any $P\in W$.    
On the other hand, since $Q$ does not belongs to 
$\overline{B(\bigcap_{P\in W}H(P), r)}=\overline{B(W^{\circ}, r)}$,  
it follows that  $\overline{B(Q, r)}\cap W^{\circ}= \emptyset$.    
Since $W^{\circ}$ and $\overline{B(Q, r)}$ are closed spherical convex sets, 
by Proposition \ref{proposition matousek}, 
there exists a point $P\in S^{n}$ such that $W^{\circ}\subset H(P)$ and 
$\overline{B(Q, r)}\cap H(P)=\emptyset$.    
By Lemmas \ref{lemma 3} and \ref{lemma 4}, 
it follows that $P\in W^{\circ \circ}=W$.    
Therefore, there exists a point $P\in W$ 
such that $Q\notin \overline{B(H(P), r)}.$    
This contradicts 
the assumption $Q \in \bigcap_{P\in W}\overline{B(H(P), r)}$. 
}
\end{proof}
\subsection{Pompeiu-Hausdorff metric}\label{pompeiu-hausdorff}
\par 
\begin{definition}[\cite{barnsley}]\label{Pompeiu-Hausdorff}
{\rm Let $(X, d)$ be a complete metric space.     
\begin{enumerate}
\item Let $x$ {\color{black}(resp., $B$)} be a point of 
$X$ {\color{black}(resp., a non-empty compact subset 
of $X$)}.     
Define 
\[
d(x,B)=\min\{d(x,y)\; |\; y\in B\}.
\] 
Then, $d(x,B)$ is called the {\it distance from the point $x$ to the set $B$}.   
\item Let $A, B$ be two non-empty compact subsets of $X$.   
Define 
\[
d(A,B)=\max \{d(x, B)\; |\; x\in A\}.   
\] 
Then, $d(A,B)$ is called the {\it distance from the set $A$ to the set $B$}.  
\item Let $A, B$ be two non-empty compact subsets of $X$.   
Define 
\[
h(A,B)=\max \{d(A, B), d(B, A)\}.   
\] 
Then, $h(A,B)$ is called the {\it Pompeiu-Hausdorff distance between $A$ and $B$}.  
\end{enumerate}
}
\end{definition}
\par 
Let $A, B$ be two non-empty compact subsets 
of a complete metric space $(X,d)$.    
The Pompeiu-Hausdorff distance between $A$ and $B$ naturally induces the 
{\it Pompeiu-Hausdorff metric} 
$h: \mathcal{H}(X)\times \mathcal{H}(X)\to \mathbb{R}_+\cup \{0\}$, 
where $\mathcal{H}(X)$ is the set consisting of non-empty 
compact sets of $X$ and $\mathbb{R}_+$ is the set of positive 
real numbers.   
The set $\mathcal{H}(X)$  
is a metric space with respect to the Pompeiu-Hausdorff metric.    
It is well-known that the metric space 
$(\mathcal{H}(X), h)$ is complete.    For details on $(\mathcal{H}(X), h)$, see 
for example \cite{barnsley, falconer}.       
\subsection{Lipschitz mappings}
{\color{black}
\begin{definition}\label{Lipschitz}
{\rm 
Let $(X, d_X), (Y,d_Y) $ be metric spaces.   
A mapping $f: X\to Y$ is said to be {\it Lipschitz} if 
there exists a positive real number $K\in \mathbb{R}$ such that 
the following holds for any $x_1, x_2\in X$:  
\[
d_Y(f(x_1), f(x_2))\le K d_X(x_1, x_2).   
\]    
The positive real number $K\in \mathbb{R}$ for a Lipschitz 
mapping is called the {\it Lipschitz coefficient} of $f$.    
}
\end{definition}
}
{\color{black}
\begin{definition}\label{bi-Lipschitz isometry}
{\rm 
Let $(X, d_X), (Y, d_{\color{black}Y})$ be metric spaces.
\begin{enumerate}   
\item 
A mapping $f: X\to Y$ is said to be {\it bi-Lipschitz} if 
$f$ is bijective and there exist positive real numbers $K, L\in \mathbb{R}$ 
such that 
the following hold for any $x_1, x_2\in X$ and any $y_1, y_2\in Y$:  
\begin{eqnarray*}
d_Y(f(x_1), f(x_2)) & \le & K d_X(x_1, x_2), \\ 
d_X(f^{-1}(y_1), f^{-1}(y_2)) & \le & L d_Y(y_1, y_2),   
\end{eqnarray*}    
\item 
A mapping $f: X\to Y$ is {\color{black}called} an {\it isometry} if 
$f$ is bijective and the following holds for any $x_1, x_2\in X$:  
\[
d_Y(f(x_1), f(x_2)) = d_X(x_1, x_2). 
\]
\end{enumerate}
}
\end{definition}
}
\begin{proposition}\label{proposition 1} 
For any $n\in \mathbb{N}$, 
the spherical polar transform 
$\bigcirc: \mathcal{H}^\circ (S^n)\to \mathcal{H}^\circ(S^n)$ 
is Lipschitz with respect to the Pompeiu-Hausdorff metric.    
\end{proposition}
\begin{proof}
We first show Proposition \ref{proposition 1} 
under the assumption that  
$W^{\circ}$ is a spherical convex set. 
Suppose that $\bigcirc$ is not Lipschitz. Then, for any $K>0$ there exist $W_{1}, W_{2}\in \mathcal{H}^{\circ}(S^{n})$ such that $Kh(W_1,W_2)<h(W_{1}^{\circ },W_{2}^{\circ })$.     
In particular, for $K=2$ there exist $W_{1}, W_{2}\in \mathcal{H}^{\circ}(S^{n})$ such that $2h(W_1,W_2)<h(W_{1}^{\circ },W_{2}^{\circ }).$ Since $h(X, Y)\leq \pi$ for any $X, Y\in \mathcal{H}^{\circ}(S^{n})$, {\color{black} i}t follows that $h(W_1,W_2)<\frac{\pi }{2}$.    
Set $r=h(W_1,W_2).$ Then, since $2r=2h(W_1,W_2)<h(W_{1}^{\circ },W_{2}^{\circ }),$ by Definition \ref{Pompeiu-Hausdorff}, it follows that at least one of $d(W_{1}^{\circ },W_{2}^{\circ })>2r$ and $d(W_{2}^{\circ },W_{1}^{\circ })>2r$ holds. Therefore, at least one of the following two holds.\vspace{2mm}\\
(1) There exists a point $P\in W_{1}^{\circ }$ such that $d(P,Q)>2r$ for any $Q\in W_{2}^{\circ }.$\vspace{2mm}\\
(2) There exists a point $Q\in W_{2}^{\circ }$ such that $d(Q,P)>2r$ for any $P\in W_{1}^{\circ }.$\vspace{2mm}\\
We show that (1) implies a contradiction. Suppose that there exists a point $\widetilde{P}\in W_{1}^{\circ}$ such that $\widetilde{P}\notin \overline{B(W_{2}^{\circ}, 2r)}$. 
In particular, $\widetilde{P}$ does not belong to $\overline{B(W_{2}^{\circ}, r)}$. 
{\color{black} Notice that the assumption that $W_{2}^{\circ}$ is a spherical convex set. Notice furthermore that $r$ is less than $\frac{\pi}{2}$. 
Thus by Lemma \ref{lemma 2},} 
we have the following: 
\[
\widetilde{P} \notin \overline{B(W_{2}^{\circ}, r)}
=\overline{B\left(\bigcap_{Q\in W_{2}}H(Q), r\right)}
=\bigcap_{Q\in W_{2}}\overline{B(H(Q), r)}. 
\]
Hence, there exists a point $Q \in W_{2}$ 
such that $\widetilde{P}\notin \overline{B(H(Q), r)}$.\\
\indent On the other hand, since $h(W_{1}, W_{2})=r$, 
there exists a point $P_{Q}\in W_{1}$ such that $d(P_{Q}, Q)\leq r$. 
Thus, by Lemma \ref{lemma 1}, 
it follows that 
$\widetilde{P}\in H(P_{Q})\subset \overline{B(H(Q), r)}.$    
Therefore, we have a contradiction.\\
\indent In the same way, we can show that (b) implies a contradiction. \\
\indent
Next we show that for any closed 
$W, \widetilde{W}\in \mathcal{H}^{\circ}(S^{n})$ such that at least one of 
$W^{\circ}, \widetilde{W}^{\circ}$ {\color{black} is not spherical} convex, 
{\color{black} the inequality }
$h(W^{\circ}, \widetilde{W}^{\circ})\leq 2h(W, \widetilde{W})$ holds.    
Since $W, \widetilde{W}\in \mathcal{H}^{\circ}(S^{n})$, there exists 
$P, \widetilde{P}\in S^{n}$ such that 
$W\subset H(P), \widetilde{W}\subset H(\widetilde{P})$. 
{\color{black} Set} $W_{i}
=
\overline{B(W, \frac{1}{i})}
\cap 
\overline{B(H(P), \frac{\pi}{2}-\frac{1}{i})},\  \widetilde{W}_{i}
=
\overline{B(\widetilde{W}, \frac{1}{i})}
\cap 
\overline{B(H(\widetilde{P}), \frac{\pi}{2}-\frac{1}{i})}$ for any $i\in \mathbb{N}$.
Since {\color{black} both} $W_{i}^{\circ}, \widetilde{W}_{i}^{\circ}$ 
are {\color{black} spherical convex, by the proof given above}, 
we have that 
$h(W_{i}^{\circ}, \widetilde{W}_{i}^{\circ})\leq {\color{black} 2}h(W_{i}, \widetilde{W}_{i})$ 
 for any $i\in \mathbb{N}$. 
{\color{black} Notice that} 
$W=\lim_{i\to \infty}W_{i},\ \widetilde{W}=\lim_{i\to \infty}\widetilde{W}_{i}$. 
Therefore, for any $i\in \mathbb{N}$, it follows that
\begin{eqnarray*}
h(W^{\circ}, \widetilde{W}^{\circ}) 
&\leq & 
h(W^{\circ}, W_{i}^{\circ})+h(W_{i}^{\circ}, 
 \widetilde{W}_{i}^{\circ})+h(\widetilde{W}_{i}^{\circ}, \widetilde{W}^{\circ})\\
&\leq & 
h(W^{\circ}, W_{i}^{\circ})+2h(W_{i}, \widetilde{W}_{i})
+
h(\widetilde{W}_{i}^{\circ}, \widetilde{W}^{\circ}).
\end{eqnarray*}
In \cite{aperture}, it has been shown that 
$\bigcirc: \mathcal{H}^\circ(S^2)\to \mathcal{H}^\circ(S^2)$ is continuous.   
It is easily seen that the proof of this result given in \cite{aperture} 
works well for general $n\in \mathbb{N}$.   
Thus, we have that
$\lim_{i\to \infty}h(W^{\circ}, W_{i}^{\circ})=0$, 
$\lim_{i\to \infty}h(\widetilde{W}^{\circ}, \widetilde{W}_{i}^{\circ})=0$.    
Therefore, we have the following:
\[
h(W^{\circ}, \widetilde{W}^{\circ})
\leq 
2\lim_{i\to \infty}h(W_{i}, \widetilde{W}_{i})=2h(W, \widetilde{W}).
\]
\noindent
Thus, the spherical polar transform 
$\bigcirc: \mathcal{H}^\circ (S^n)\to \mathcal{H}^\circ(S^n)$ 
is Lipschitz. 
\end{proof}
\par
\begin{claim} \label{claim 1}
The following example {\color{black} shows} that the natural number \lq \lq \ $2$ \rq \rq is the least real number for the Lipschitz coefficient 
{\color{black}of $\bigcirc$}.
\end{claim}
\par
\noindent
{\bf Example:}\ \ For any real number $r\ (1<r<2)$, there exist a real number $r_{1}$ 
and two points $P_{1}, P_{2}\in S^{n}$ such that $r\frac{\pi}{2}<r_{1}<\pi$ and $d(P_{1}, P_{2})=r_{1}$. 
{\color{black} Since $H(P_{i})\subset S^{n}=
\overline{B\left(H(P_{j}), \frac{\pi}{2}\right)},\ \{i, j\}=\{1, 2\}$,
}
we have that $h(H(P_{1}), H(P_{2}))
{\color{black}\leq}
\frac{\pi}{2}$. 
Set $W_{1}=H(P_{1}), W_{2}=H(P_{2})$. Then, we have the following:
\[
rh(W_{1}, W_{2})
{\color{black}\leq}
\ r\frac{\pi}{2}
{\color{black} <}\ r_{1}=d( P_{1}, P_{2}) = h(\{P_{1}\}, \{P_{2}\})=h(W_{1}^{\circ}, W_{2}^{\circ}).
\]
It follows the following({\color{black}s}ee {\color{black} F}igure \ref{figure 1}):
\[
rh(W_{1}, W_{2})<h(W_{1}^{\circ}, W_{2}^{\circ}).
\]
\begin{figure}[hbtp]
\begin{center}
\includegraphics[width=4cm]{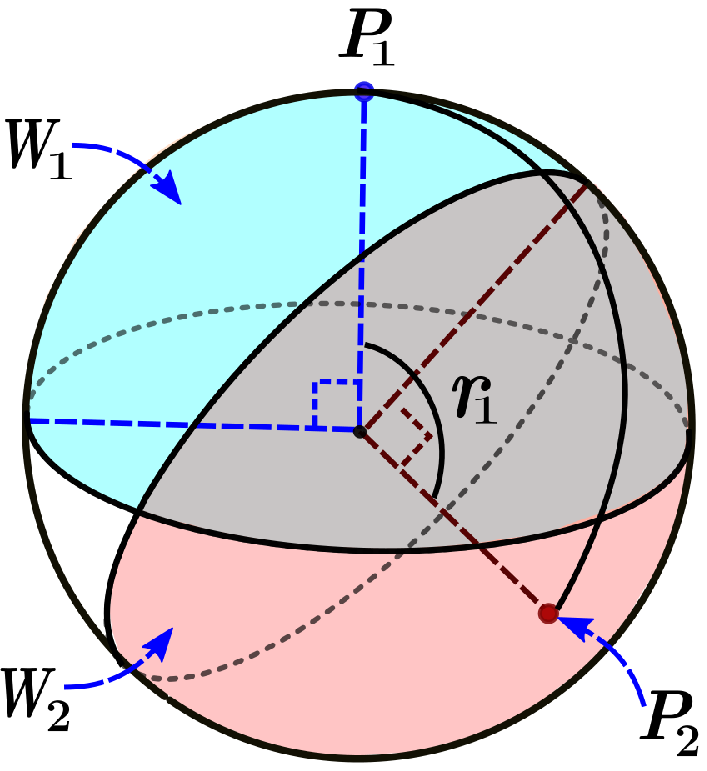}
\caption{$r\ h(W_{1}, W_{2})
<h({\color{black}\{{\color{black}P_1}\}, \{{\color{black}P_2}\}})\ (1<r<2)$.}
\label{figure 1}
\end{center}
\end{figure}
\par
By Proposition \ref{proposition 1}, we can extend Lemma \ref{lemma 4} as follows.
\begin{lemma}\label{lemma 5}
For any $X=\lim_{i\to \infty} X_{i}$, the following equality holds. 
\[ 
X=X^{\circ\circ},
\]
where $X_{i}\in \mathcal{H}_{\mbox{\rm s-conv}}(S^{n})$ \quad $ (i= 1, 2, 3, \dots)$.
\end{lemma}
\begin{proof} 
By Proposition \ref{proposition 1}, it follows that 
the composition $\bigcirc \circ \bigcirc$ is Lipschitz.    
Thus, if $X=\lim_{i\to \infty} X_i$ then $X^{\circ \circ}=\lim_{i\to \infty} X_{i}^{\circ \circ}$. By Lemma \ref{lemma 4}, $X=\lim_{i\to \infty} {X}_{i}=\lim_{i\to \infty} {X}_{i}^{\circ \circ}=X^{\circ \circ}$.
\end{proof}
\section{Proof of Theorem \ref{theorem 1}}\label{section 3}
\indent 
We first show that for any $W_{1}, W_{2}\in \mathcal{H}_{{\rm Wulff}}(S^{n}, P)$, the following holds:
\[
h(W_{1}^{\circ}, W_{2}^{\circ})\leq h(W_{1}, W_{2}).
\]
\par
\par
Suppose that there exist $W_{1}, W_{2}\in \mathcal{H}_{{\rm Wulff}}(S^{n}, P)$ such that $h(W_{1}, W_{2})< h(W_{1}^{\circ}, W_{2}^{\circ})$. Since $h(X, Y)< \frac{\pi}{2}$ for any $X, Y\in \mathcal{H}_{{\rm Wulff}}(S^{n}, P)$, it follows that $h(W_{1}, W_{2})< \frac{\pi}{2}$. Set $r=h(W_{1}, W_{2})$. Then, since $r=h(W_{1}, W_{2})< h(W_{1}^{\circ}, W_{2}^{\circ})$, 
it follows that at least one of $d(W_{1}^{\circ },W_{2}^{\circ })>r$ and $d(W_{2}^{\circ },W_{1}^{\circ })>r$ holds, where $d(A, B)$ is the distance from $A$ to $B$ defined in Section 2. Therefore, at least one of the following two holds.\\
(a) There exists a point $Q_{1}\in W_{1}^{\circ }$ such that $d(Q_{1}, R_{2})>r$ for any $R_{2}\in W_{2}^{\circ }.$\vspace{2mm}\\
(b) There exists a point $Q_{2}\in W_{2}^{\circ }$ such that $d(Q_{2}, R_{1})>r$ for any $R_{1}\in W_{1}^{\circ }.$\vspace{2mm}\\
\indent Suppose that (a) holds. This means that there exists a point $Q_{1}\in W_{1}^{\circ}$ such that $Q_{1}\notin \overline{B(W_{2}^{\circ}, r)}$. Then, by Lemma \ref{lemma 2}, we {\color{black} have the following:}
\[
Q_{1}\notin \overline{B(W_{2}^{\circ}, r)}=\overline{B\left(\bigcap_{\widetilde{Q}\in W_{2}}H(\widetilde{Q}), r\right)}=\bigcap_{\widetilde{Q}\in W_{2}}\overline{B(H(\widetilde{Q}), r)}. 
\]
Hence, there exists a point $R\in W_{2}$ such that $Q_{1}\notin \overline{B(H(R), r)}$.\\
\indent 
On the other hand, since $h(W_{1}, W_{2})=r$, 
there exists a point $\widetilde{P}_{R}\in W_{1}$ 
such that $d(\widetilde{P}_{R}, R)\leq r$. 
By Lemma \ref{lemma 1}, 
$Q_{1}\in H(\widetilde{P}_{R})\subset \overline{B(H(R), r)}.$ 
Therefore, we have a contradiction.\\
\indent In the same way, we can show that (b) implies a contradiction.\\
Therefore, we have proved that {\color{black}for any $W_{1}, W_{2}\in \mathcal{H}_{{\rm Wulff}}(S^{n}, P)$, the following holds:
\[
h(W_{1}^{\circ}, W_{2}^{\circ})\leq h(W_{1}, W_{2}).
\]}
\par
\indent
{\color{black} B}y Lemma \ref{lemma 4}, for any $W_{1}, W_{2}\in \mathcal{H}_{\rm Wulff}(S^{n}, P)$ we have the following:
\[
h(W_{1}, W_{2})\leq h(W_{1}^{\circ}, W_{2}^{\circ})\leq h(W_{1}, W_{2}).
\]
{\color{black}
Therefore, we have that $h(W_{1}, W_{2})=h(W_{1}^{\circ}, W_{2}^{\circ})$ for any 
$W_{1}, W_{2}\in \mathcal{H}_{{\rm Wulff}}(S^{n}, P)$.  
}
\par
\indent
(2) Let $W_{1}=\lim_{i\to \infty}W_{1_{i}}, W_{2}=\lim_{i\to \infty}W_{2_{i}}$, 
where $W_{1_{i}}, W_{2_{i}}\in \mathcal{H}_{\rm Wulff}(S^{n}, P)$ 
for any $i\in \mathbb{N}$. {\color{black} By (1), 
we have that} $h(W_{1_{i}}, W_{2_{i}})=h(W_{1_{i}}^{\circ}, W_{2_{i}}^{\circ})$. 
By Proposition \ref{proposition 1}, we have that 
\begin{eqnarray*}
h(W_{1}, W_{2})& = & h(\lim_{i\to \infty}W_{1_{i}}, \lim_{i\to \infty}W_{2_{i}})
=\lim_{i\to \infty}h(W_{1_{i}}, W_{2_{i}})\\
 & = & \lim_{i\to \infty}h(W_{1_{i}}^{\circ}, W_{2_{i}}^{\circ})
= h(\lim_{i\to \infty}W_{1_{i}}^{\circ}, \lim_{i\to \infty}W_{2_{i}}^{\circ})
=h(W_{1}^{\circ}, W_{2}^{\circ}). 
\end{eqnarray*}
\hfill {$\Box$} \\
\section{Proof of Theorem \ref{theorem 2}}\label{section 4}
By the proof of Proposition \ref{proposition 1}, we have that 
\begin{eqnarray*}
h(W_{1}^{\circ},W_{2}^{\circ}) & \leq & 2h(W_{1},W_{2}) \\
h(W_{1}^{\circ \circ},W_{2}^{\circ \circ}) & \leq & 2 h(W_{1}^{\circ},W_{2}^{\circ})
\end{eqnarray*}
for any $W_{1}, W_{2}\in \mathcal{H}_{\mbox{\rm s-conv}}(S^{n})$. \\
By Lemma \ref{lemma 5}, we have {\color{black} the following 
for any $W_{1}, W_{2}\in \overline{\mathcal{H}_{\mbox{\rm s-conv}}(S^{n})}$:}
\[
W_{1}^{\circ \circ} =  W_{1},  W_{2}^{\circ \circ}  =  W_{2}.
\]
Therefore, the following inequality holds for any 
$W_{1}, W_{2}\in \overline{\mathcal{H}_{\mbox{\rm s-conv}}(S^{n})}$:
\[
\frac{1}{2}h(W_{1},W_{2})\leq h(W_{1}^{\circ},W_{2}^{\circ})\leq 2h(W_{1},W_{2}).
\] 
\par
Hence, $\bigcirc _{\mbox{\rm s-conv}}: 
\overline{\mathcal{H}_{\mbox{\rm s-conv}}
(S^{n})}\rightarrow \overline{\mathcal{H}_{\mbox{\rm s-conv}}
(S^{n})}$ is bi-{\color{black} L}ipschitz. \\
\indent
By Claim \ref{claim 1}, it is clear that $\bigcirc_{\mbox{\rm s-conv}}: 
\overline{\mathcal{H}_{\mbox{\rm s-conv}}(S^{\circ})}\rightarrow 
\overline{\mathcal{H}_{\mbox{\rm s-conv}}(S^{\circ})}$ is 
never isometric. \hfill {$\Box$}\\

\section{Appendix}\label{section 5}
\subsection{Mappings in Theorems are well-defined bijections}
\label{subsection 5.1}
{\color{black}
\begin{proposition}\label{proposition 2}
\begin{enumerate}
\item 
For any point $P\in S^n$, 
 $\overline{\mathcal{H}_{\rm Wulff}(S^n, P)}\subset \mathcal{H}^\circ(S^n)$.  
\item 
For any point $P\in S^n$, 
$\bigcirc({\mathcal{H}_{\rm Wulff}(S^n, P)})
= 
{\mathcal{H}_{\rm Wulff}(S^n, P)}$.   
\item 
For any point $P\in S^n$, 
$\bigcirc(\overline{\mathcal{H}_{\rm Wulff}(S^n, P)})
= 
\overline{\mathcal{H}_{\rm Wulff}(S^n, P)}$.   
\item 
For any point $P\in S^n$, 
the restriction of $\bigcirc$ to 
$\overline{\mathcal{H}_{\rm Wulff}(S^n, P)}$  
is injective.  
\item 
$\overline{\mathcal{H}_{\mbox{\rm s-conv}}(S^n)}\subset \mathcal{H}^\circ(S^n)$.  
\item 
$\bigcirc({\mathcal{H}_{\mbox{\rm s-conv}}(S^n)})
\ne  
{\mathcal{H}_{\mbox{\rm s-conv}}(S^n)}$.   
\item 
$\bigcirc(\overline{\mathcal{H}_{\mbox{\rm s-conv}}(S^n)})
= 
\overline{\mathcal{H}_{\mbox{\rm s-conv}}(S^n)}$.   
\item 
The restriction of $\bigcirc$ to 
$\overline{\mathcal{H}_{\mbox{\rm s-conv}}(S^n)}$ 
is injective.   
\end{enumerate}
\end{proposition}
}
\noindent 
{\bf Proof of Proposition \ref{proposition 2}. }
\quad{}  
\par 
{\it Proof of the assertion (1) of Proposition \ref{proposition 2}.}
\quad 
 {\color{black}It is clear that f}or any $W\in \overline{\mathcal{H}_{\rm Wul{\color{black}f}f}(S^{n}, P)}$,  
we have that $W\subset H(P)$. Thus, by Lemma \ref{lemma 3}, it follows that $P\in W^{\circ}$, {\color{black} which implies $W^{\circ}\neq \emptyset$}.   
\hfill {$\Box$}\\
\par 
\smallskip 
{\it Proof of the assertion (2) of Proposition \ref{proposition 2}.}
\quad 
{\color{black}    We first show that 
$\bigcirc (W)\in \mathcal{H}_{\rm Wulff}(S^{n}, P)$ for any $W\in \mathcal{H}_{\rm Wulff}(S^{n}, P)$.    
For any $W\in \mathcal{H}_{\rm Wulff}(S^{n}, P)$, 
there exist $r_{1}, r_{2}$\ $(0<r_{1}<r_{2}<\frac{\pi}{2})$ such that $\overline{B(P, r_{1})}\subset W \subset \overline{B(P, r_{2})}$.    
By Lemma \ref{lemma 3}, we have the following:
\[
\left(\overline{B(P, r_{2})}\right)^{\circ}\subset W^{\circ}\subset \left(\overline{B(P, r_{1})}\right)^{\circ}.
\]
It follows that $W^{\color{black}\circ}\cap H(-P)
=\emptyset$ and $P\in \mbox{\rm int}(W^{\color{black}\circ})$.   
Let $Q_1, Q_2$ be two points of $W^\circ=\bigcap_{Q\in W}H(Q)$.   
Since $W^{\color{black}\circ}\cap H(-P)=\emptyset$, it follows that 
$(1-t)Q_1+tQ_2$ is not the zero vector for any $t\in [0,1]$.    
Thus, for any $t\in [0,1]$ we have the following:   
\[
\frac{(1-t)Q_1+tQ_2}{||(1-t)Q_1+tQ_2 ||} \in \bigcap_{Q\in W}H(Q)=W^\circ.  
\]   
It follows that $W^\circ$ is spherical convex.   
Therefore, $\bigcirc (W)$ is contained in $\mathcal{H}_{\rm Wulff}(S^{n}, P)$.  
\par 
Next, we show that for any $W\in \mathcal{H}_{\rm Wulff}(S^{n}, P)$, 
there exists an element $\widetilde{W}\in \mathcal{H}_{\rm Wulff}(S^{n}, P)$ 
such that $\bigcirc (\widetilde{W})=W$.     
For any $W\in \mathcal{H}_{\rm Wulff}(S^{n}, P)$, 
set $\widetilde{W}=W^\circ$.  
Then, we have already proved that 
$\widetilde{W}$ is contained in $\mathcal{H}_{\rm Wulff}(S^{n}, P)$.  
By Lemma \ref{lemma 4}, 
it follows that $\bigcirc (\widetilde{W})=W$.   
\hfill {$\Box$}
\par 
\smallskip 
{\it Proof of the assertion (3) of Proposition \ref{proposition 2}.}
\quad 
{\color{black}    
We first show that 
$\bigcirc (W)\in \overline{\mathcal{H}_{\rm Wulff}(S^{n}, P)}$ 
for any $W\in \overline{\mathcal{H}_{\rm Wulff}(S^{n}, P)}$.    
For any $W\in \overline{\mathcal{H}_{\rm Wulff}(S^{n}, P)}$, 
there exists a convergent sequence 
$W_i\in \mathcal{H}_{\rm Wulff}(S^{n}, P)$ such that 
$W=\lim_{i\to \infty}W_i$.    
Since 
$\bigcirc : \mathcal{H}^\circ (S^n)\to \mathcal{H}^\circ (S^n)$ is continuous, 
it follows that $W^\circ =\lim_{i\to \infty}W_i^\circ$.   
Since we have already proved the assertion (2) of Proposition \ref{proposition 2}, 
it follows that $\bigcirc (W)\in \overline{\mathcal{H}_{\rm Wulff}(S^{n}, P)}$.    
\par 
Next, we show that for any $W\in \overline{\mathcal{H}_{\rm Wulff}(S^{n}, P)}$, 
there exists an element 
$\widetilde{W}\in \overline{\mathcal{H}_{\rm Wulff}(S^{n}, P)}$ 
such that $\bigcirc (\widetilde{W})=W$.     
For any $W\in \overline{\mathcal{H}_{\rm Wulff}(S^{n}, P)}$, 
set $\widetilde{W}=W^\circ$.  
Then, we have already proved that 
$\widetilde{W}$ is contained in $\overline{\mathcal{H}_{\rm Wulff}(S^{n}, P)}$.  
By Lemma \ref{lemma 5}, 
it follows that $\bigcirc (\widetilde{W})=W$.   
\hfill {$\Box$}
\par 
\smallskip 
{\it Proof of the assertion (4) of Proposition \ref{proposition 2}.}
\quad 
Suppose that there exist $W_{1}=\lim_{i\to \infty}W_{1_{i}}, W_{2}=\lim_{i\to \infty} W_{2_{i}}\in \overline{\mathcal{H}_{\rm Wulff}(S^{n}, P)}$ such that $W_{1}^{\circ}=W_{2}^{\circ}$,
 where $W_{1_{i}}, W_{2_{i}}\in \mathcal{H}_{\rm W{\color{black}u}lff}(S^{n}, P)$ for any $i\in \mathbb{N}$.
 Since $W_{1_{i}}, W_{2_{i}}$ are spherical convex, by Lemma \ref{lemma 5} the following holds:
\[
W_{1}=W_{1}^{\circ \circ}=W_{2}^{\circ \circ}=W_{2}.
\]
Therefore, for any point $P\in S^{n}$, the restriction of $\bigcirc$ to $\overline{\mathcal{H}_{\rm Wulff}(S^{n}, P)}$ is injective.
\hfill {$\Box$}
\par 
\smallskip 
{\it Proof of the assertion (5) of Proposition \ref{proposition 2}.}
\quad 
Let $W$ be an element of $\overline{\mathcal{H}_{\mbox{\rm s-conv}}(S^n)}$.   
Then, by Proposition \ref{relation}, there exists a point $P\in S^n$ 
such that $W\in \overline{\mathcal{H}_{\rm Wulff}(S^n, P)}$.    
Since we have already proved the assertion (1) of Proposition \ref{proposition 2}, 
it follows that $\bigcirc (W)\in \mathcal{H}^\circ (S^n)$.   
\hfill {$\Box$}
\par 
\smallskip 
{\it Proof of the assertion (6) of Proposition \ref{proposition 2}.} 
\quad 
For any point $P\in S^{n}$, 
$\bigcirc({\color{black} \{}P{\color{black} \}})=H(P)$ is not hemispherical.   
Therefore, 
$\bigcirc(\mathcal{H}_{\mbox{\rm s-conv}}(S^{n}))\neq 
\mathcal{H}_{\mbox{\rm s-conv}}(S^{n})$. \hfill {$\Box$}\\
\par 
\smallskip 
{\it Proof of the assertion (7) of Proposition \ref{proposition 2}.}
\quad 
{\color{black} We first show that 
$\bigcirc(W)\in \overline{\mathcal{H}_{\mbox{\rm s-conv}}(S^{n})}$ 
for any $W\in \overline{\mathcal{H}_{\mbox{\rm s-conv}}(S^{n})}$. 
For any $W\in \overline{\mathcal{H}_{\mbox{\rm s-conv}}(S^{n})}$,
 by Proposition \ref{relation}, 
$W\in \bigcup_{P\in S^{n}}\overline{\mathcal{H}_{\rm Wulff}(S^{n}, P)}$. 
It follows that there exists spherical convex body subsequence 
$\{\widetilde{W}_{i}\}_{i=1, 2, 3, \dots}$ such that 
$W=\lim_{i\rightarrow \infty}\widetilde{W}_{i}$, 
where $\widetilde{W}_{i}\in \mathcal{H}_{\rm Wulff}(S^{n}, P_{i})$. 
Since $\bigcirc$ is continuous and 
$\widetilde{W}_{i}^{\circ} \in \overline{\mathcal{H}_{\mbox{\rm s-conv}}(S^{n})}$ 
($i=1, 2, 3, \dots$), 
we have that $\lim_{n\to \infty}\widetilde{W}_{i}^{\circ}
=W^{\circ}\in \overline{\mathcal{H}_{\mbox{\rm s-conv}}(S^{n})}$.\\
\indent
Next, we show that for any $W\in \overline{\mathcal{H}_{\mbox{\rm s-conv}}(S^{n})}$,
 there exists an element 
$\widetilde{W}\in \overline{\mathcal{H}_{\mbox{\rm s-conv}}(S^{n})}$ 
such that $\bigcirc(\widetilde{W})=W$.
 For any $W \in \overline{\mathcal{H}_{\mbox{\rm s-conv}}(S^{n})}$, 
set $\widetilde{W}=W^{\circ}$. By Lemma \ref{lemma 5}, 
it follows that $W=\bigcirc(W^{\circ})$.}\hfill $\square$\\
\par 
\smallskip 
{\it Proof of the assertion (8) of Proposition \ref{proposition 2}.}
\quad 
Suppose that there exist $W_{1}=\lim_{i\to \infty}W_{1_{i}}, 
W_{2}=\lim_{i\to \infty} W_{2_{i}}\in \overline{\mathcal{H}_{\mbox{\rm s-conv} }(S^{n})}$ 
such that $W_{1}^{\circ}=W_{2}^{\circ}$, 
where $W_{1_{i}}, W_{2_{i}}\in \mathcal{H}_{\mbox{\rm s-conv}}(S^{n})$ 
for any $i\in \mathbb{N}$. Since $W_{1_{i}}, W_{2_{i}}$ 
are spherical convex, by Lemma \ref{lemma 5} the following holds:
\[
W_{1}=W_{1}^{\circ \circ}=W_{2}^{\circ \circ}=W_{2}.
\]
Therefore, for any point $P\in S^{n}$, 
the restriction of $\bigcirc$ to $\overline{\mathcal{H}_{\mbox{\rm s-conv}}(S^{n})}$ 
is injective.\hfill {$\Box$}\\
}
}
By Proposition \ref{proposition 2}, {\color{black}we have the following:   
\begin{proposition}\label{lemma 1.1}
Each of the following is {\color{black}a} well-defined bijective mapping.   
\begin{eqnarray*}
\bigcirc_{{\rm Wulff}, P}: 
{\mathcal{H}_{\rm Wulff}(S^n, P)} 
& \to &  
{\mathcal{H}_{\rm Wulff}(S^n, P)},  \\ 
\overline{\bigcirc_{{\rm Wulff}, P}}: 
\overline{\mathcal{H}_{\rm Wulff}(S^n, P)} 
& \to &  
\overline{\mathcal{H}_{\rm Wulff}(S^n, P)},  \\ 
\bigcirc_{\mbox{\rm s-conv}}: \overline{\mathcal{H}_{\mbox{\rm s-conv}}(S^n)} 
& \to &  
\overline{\mathcal{H}_{\mbox{\rm s-conv}}(S^n)}.    \\  
\end{eqnarray*} 
\end{proposition}
}
Moreover, by the assertion (6) of Proposition \ref{proposition 2}, 
it follows that the restriction of $\bigcirc$ to $\mathcal{H}_{\mbox{s-conv}}(S^n)$ is not a transform.      Thus, the restriction to this subspace 
is not investigated in this paper. 
\subsection{Why the restriction of $\bigcirc$ to $\overline{\mathcal{H}_{\mbox{s-conv}}(S^n)}$ is important ?}
\label{subsection 5.2}
\quad {}
\par 
It is natural to expect that the isometric property still holds 
even for the restriction of $\bigcirc$ to  
$\bigcup_{P\in S^n}\overline{\mathcal{H}_{\rm Wulff}(S^n, P)}$.   
Since this subspace of $\mathcal{H}^\circ(S^n)$ seems to be 
complicated, we want to have a translation of this space into a subspace of $\mathcal{H}^\circ(S^n)$ which is easy to treat.     
The following proposition is the desired translation.   
Hence, the subspace $\overline{\mathcal{H}_{\mbox{s-conv}}(S^n)}$ 
naturally arises in our study and the restriction of $\bigcirc$ to $\overline{\mathcal{H}_{\mbox{s-conv}}(S^n)}$ is important.    
{\color{black}
\begin{proposition}\label{relation}
\[
\bigcup_{P\in S^n}\overline{\mathcal{H}_{\rm Wulff}(S^n, P)}   
=\overline{\mathcal{H}_{\mbox{\rm s-conv}}(S^n)}. 
 \]
\end{proposition}
}
\begin{proof}
\quad 
By Definition \ref{definition closure}, it is clear that any $W\in \bigcup_{P\in S^n}\overline{\mathcal{H}_{{\rm Wulff}}(S^n, P)} $ is an element of $\overline{\mathcal{H}_{\mbox{\rm s-conv }}(S^{n})}$. {\color{black} Thus, it is sufficient to} show the following inclusion:   
\[
\overline{\mathcal{H}_{\mbox{\rm s-conv }}(S^{n})}
\subset 
\bigcup_{P\in S^n}\overline{\mathcal{H}_{\rm Wulff}(S^n, P)}.  
\]
\par
We first show the above inclusion 
under the assumption that 
%
$W\in \overline{\mathcal{H}_{\mbox{\rm s-conv }}(S^{n})}$ 
is a hemispherical closed subset of $S^{n}$. 
Suppose {\color{black} that} $W$ has an interior point. 
Then, it is easily seen that there exists a point $P\in {\rm int}(W)$ {\color{black} such that} 
$W\subset H(P)$. 
Since $H(P)\in \overline{\mathcal{H}_{\rm Wulff}(S^{n}, P)}$, 
it follows that $W\in \overline{\mathcal{H}_{\rm Wulff}(S^{n}, P)}$.
Next, suppose that $W$ does not have an interior point. 
Since $W$ is hemispherical and closed, 
there exist a point $P\in W$ 
and a positive {\color{black} integer} number $N\in \mathbb{N}$ 
such that for any $i>N$, 
we have 
${\color{black} \partial\overline{B(W, \frac{2}{i})}\bigcap H(-P)=\emptyset}$.    
For any $i>N$,  
there exists a sequence 
\[
\{W_{i}\}_{i= 1, 2, 3, \dots}\subset \mathcal{H}_{\mbox{\rm s-conv}}(S^{n})
\]  
such that 
$h(W_{i}, W)<\frac{1}{i}$. 
Thus, we have the following:
\[ 
P\in W\subset \overline{B\left(W_{i}, \frac{1}{i}\right)}
\subset 
\overline{B\left(\overline{B\left(W, \frac{1}{i}\right)}, \frac{1}{i}\right)}
=\overline{B\left(W, \frac{2}{i}\right)}\subset H(P). 
\]
Therefore, {\color{black} it follows that for 
$\overline{B\left(W_{i}, \frac{1}{i}\right)}\in \mathcal{H}_{{\rm Wulff}}(S^{n}, P)$, 
we have the following:} 
\[
W=\lim_{i\to \infty}\overline{B\left(W_{i}, \frac{1}{i}\right)}\in 
\bigcup_{P\in S^{n}}\overline{\mathcal{H}_{\rm Wulff}(S^{n}, P)}.
\] 
\indent
Finally,{\color{black} we show Proposition {\color{black}\ref{relation}} 
in the case that 
$W$ is an element of $\overline{\mathcal{H}_{\mbox{\rm s-conv}}(S^{n})}$}. 
Notice that there exists a point $P\in S^{n}$ such that $W\subset H(P)$.    
For any positive interger $i$, define $W_{i}$ as follows. 
\[ 
W_{i}=\overline{B\left(W, \frac{1}{i} \right)}
\bigcap \overline{B\left(P, \frac{\pi}{2}-\frac{1}{i} \right)}.
\]
Then, 
it is easily seen that $W_{i}\in \mathcal{H}_{\rm Wulff}(S^{n}, P)$ 
for any $i\in \mathbb{N}$ and $W=\lim_{i\to \infty} W_{i}$.    
Therefore, 
$W$ belongs {\color{black}to} 
$\bigcup_{P\in S^n}\overline{\mathcal{H}_{\rm Wulff}(S^n, P)})$.
\end{proof}



\end{document}